\documentclass[11pt,reqno]{amsart}
\usepackage[T1]{fontenc}
\usepackage{amsmath,amssymb}
\usepackage[a4paper,margin=30mm]{geometry}
\usepackage[hidelinks]{hyperref}

\newtheorem{theorem}{Theorem}[section]
\newtheorem{lemma}[theorem]{Lemma}
\newtheorem{proposition}[theorem]{Proposition}
\newtheorem{corollary}[theorem]{Corollary}
\theoremstyle{definition}
\newtheorem{definition}[theorem]{Definition}

\theoremstyle{remark}
\newtheorem{remark}[theorem]{Remark}

\newcommand{\Rn}{\mathbf R_n}
\newcommand{\Slp}{\mathbf S_{\ell}^{+}}
\newcommand{\W}{\mathbf W}
\newcommand{\V}{\mathbf V}
\newcommand{\A}{\mathbf A}
\newcommand{\B}{\mathbf B}

\newcommand{\F}{\mathbb F}
\newcommand{\Hp}{\mathcal H^{+}}
\newcommand{\eps}{\varepsilon}

\DeclareMathOperator{\Var}{Var}
\DeclareMathOperator{\Con}{Con}
\DeclareMathOperator{\Id}{Id}
\DeclareMathOperator{\Pol}{Pol}
\DeclareMathOperator{\SI}{SI}
\DeclareMathOperator{\HS}{HS}
\DeclareMathOperator{\HSP}{HSP}

\newcommand{\down}{\mathord{\downarrow}}
\newcommand{\up}{\mathord{\uparrow}}
\newcommand{\eqid}{\approx}

\numberwithin{equation}{section}
\allowdisplaybreaks[1]
\title[Mal'cev products of rings and idempotent semirings]
{Mal'cev products of rings satisfying\\
$x^n\approx x$
and idempotent semirings}
\author{Aifa Wang}
\address{School of Mathematical Sciences, Chongqing University of Technology, Chongqing 400054, P.R. China}
\email{wangaf@cqut.edu.cn}
\thanks{This work was supported by the Science and Technology Research Program of Chongqing Municipal
Education Commission (KJZD-K2024011102) and the Chongqing Natural Science Foundation Innovation
and Development Joint Fund (Municipal Education Commission) (CSTB2025NSCQ-LZX0067).}
\author{Lili Wang}
\address{School of Mathematical Sciences, Chongqing University of Technology, Chongqing 400054, P.R. China}
\email{wllaf@cqut.edu.cn}
\date{}

\begin{document}
\begin{abstract}
Let $\Rn$ be the variety of rings satisfying $x^n\eqid x$, where
$n\geq2$, and let $\W$ be a variety of idempotent semirings with
commutative addition. The equality $\Rn\circ\W=\Rn\vee\W$ holds,
and this variety has subvariety lattice
$L(\Rn)\times L(\W)$. Its subdirectly irreducible members that
are not additively idempotent have exactly one nontrivial ring
component, which is a finite field. Such a member is determined
by this field and an idempotent semiring with a multiplicative
identity; a separation condition on unary polynomial functions
of the latter characterizes
subdirect irreducibility. Every subvariety has a finite identity
basis. A four-element semiring generates a variety with
subdirectly irreducible members of unbounded cardinality.
The results extend the Mal'cev product theorem of Wang and Shao
from the absorption subvariety to arbitrary $\W\leq\Slp$.
\end{abstract}
\maketitle

\section{Introduction and preliminaries}

By a \emph{semiring} we mean an algebra $(S,+,\cdot)$ such that
$(S,+)$ is a commutative semigroup, $(S,\cdot)$ is a semigroup, and
the distributive identities
\[
 x(y+z)\eqid xy+xz,\qquad (y+z)x\eqid yx+zx
\]
hold in $S$. Neither a zero nor an identity element is included
in the similarity type. Rings and fields are regarded as
semirings of this type, and homomorphisms are not required to
preserve multiplicative identity elements. A semiring is
\emph{idempotent} if it satisfies
\[
 x+x\eqid x,\qquad x^2\eqid x.
\]
We denote the variety of all such semirings by $\Slp$.
Throughout the paper, $n\geq2$ is fixed, $\Rn$ denotes the
variety of rings satisfying $x^n\eqid x$, and
\begin{equation}\label{eq:m}
 m=2^n-2.
\end{equation}
For a positive integer $k$, the term $kx$ denotes the sum of
$k$ copies of $x$.

For classes $\A$ and $\B$ of semirings, their \emph{Mal'cev product}
$\A\circ\B$ consists of all semirings $S$ admitting a congruence
$\rho$ such that $S/\rho\in\B$ and each $\rho$-class which is a
subsemiring belongs to $\A$; this construction goes back to
Mal'cev~\cite{Malcev}. If $\B\leq\Slp$, every $\rho$-class
is a subsemiring. The join of two varieties is denoted by
$\A\vee\B$, and $L(\A)$ denotes the lattice of subvarieties of $\A$.
We write $\Var(\mathcal K)=\HSP(\mathcal K)$ and
$\Var(S)=\Var(\{S\})$.

Wang and Shao \cite{WS} studied $\Rn\circ\mathbf N^+$, where
\[
 \mathbf N^+
 =\{I\in\Slp:I\models x+xyx\eqid x\}.
\]
They proved that every member of this Mal'cev product is a
subdirect product of a ring in $\Rn$ and an idempotent semiring
in $\mathbf N^+$. Their results give
$\Rn\circ\mathbf N^+=\Rn\vee\mathbf N^+$ and a finite basis
for every subvariety. Without the absorption identity, this decomposition can fail:
the four-element semiring in Remark~\ref{rem:no-two-factor}
is a counterexample. We study what remains true
when $\mathbf N^+$ is replaced by an arbitrary subvariety
of $\Slp$.

The join equality holds for every $\W\leq\Slp$.
A member of $\Rn\circ\W$ generates the join of the varieties
generated by its ring components and its idempotent skeleton.
This description yields a direct product
decomposition of the subvariety lattice. The classification of
subdirectly irreducible members that are not additively idempotent
requires no finiteness assumption and includes members that are
not rings. Every subvariety is finitely based, whereas a
four-element semiring generates a residually large variety.

The equality congruence
on $S$ is $\Delta_S$, and $\Con(S)$ is its congruence lattice.
A nontrivial algebra is \emph{subdirectly irreducible} if it has
a least nontrivial congruence, called its \emph{monolith}.
For a variety $\V$, the class of its subdirectly irreducible
members is denoted by $\SI(\V)$, with algebras identified up to
isomorphism when appropriate. We refer to \cite{BS,Howie} for universal algebra
and semigroup terminology.

Jacobson's theorem states that rings satisfying
$x^n\eqid x$ are commutative \cite{Jacobson}.

\begin{lemma}\label{lem:rings}
The class $\Rn$, in the semiring similarity type, is defined
by the additional identities
\begin{equation}\label{eq:ring-basis}
 x^n\eqid x,\qquad (m+1)x\eqid x,\qquad mx\eqid my.
\end{equation}
Its nontrivial subdirectly irreducible members are precisely
the finite fields $\F_q$ for which $q-1$ divides $n-1$.
Every member of $\Rn$ is a subdirect product of such fields.
\end{lemma}

\begin{proof}
In a ring satisfying $x^n\eqid x$, the equality
$2a=(2a)^n=2^na$ gives $ma=0$, verifying
\eqref{eq:ring-basis}. Conversely, the last two identities
in \eqref{eq:ring-basis} give a common additive zero $mx$ and
an additive inverse $(m-1)x$ for $x$. They make the additive
reduct a group; the first identity gives membership in $\Rn$.

Let $R$ be a nontrivial subdirectly irreducible member of $\Rn$,
and let $J$ be the nonzero ideal corresponding to its monolith.
Choose $0\ne a\in J$ and put $h=a^{n-1}$. Then $h^2=h$,
$ha=a$, and $h\in J$. Since $R$ is commutative, $x\mapsto hx$
is an idempotent ring endomorphism. Its kernel does not contain
$a$, and hence is zero. An injective idempotent map is the
identity, so $h$ is the identity element of $R$.
Containing the identity, the ideal satisfies $J=R$.
This makes $R$ simple; being commutative
with identity, it is a field.

Every element of this field is a root of $X^n-X$.
The bound on the number of roots makes $R$ finite, say
$R\cong\F_q$. Its multiplicative group
is cyclic~\cite[Chapter~2]{LN}, so $q-1\mid n-1$.
Conversely, each such field belongs
to $\Rn$ and is simple. The last assertion follows from the
subdirect representation theorem~\cite[Chapter~II]{BS}.
\end{proof}

For a semiring $S$, put
\[
 E^+(S)=\{a\in S:a+a=a\},
\]
and let $\Hp$ denote Green's $\mathcal H$-relation on its
additive reduct. If $(S,+)$ is a commutative Clifford
semigroup, its $\Hp$-classes are its maximal subgroups;
see \cite{Howie}.

The variety of bands is locally finite~\cite{Howie}. Since every semiring
term is a sum of nonempty multiplicative words, $\Slp$ is
locally finite as well. The multiplicative reduct of every
member of $\Slp$ is a regular band \cite{PZ}, so the identity
\begin{equation}\label{eq:regular-band}
 xyxzx\eqid xyzx
\end{equation}
holds in $\Slp$. We do not assume that multiplication in
$\Slp$ is commutative.

\section{The Mal'cev product and the join}

The ring components of a member of $\Rn\circ\W$ are the
fibers of an additive retraction.

\begin{lemma}\label{lem:skeleton}
Let $\W\leq\Slp$. A semiring $S$ belongs to $\Rn\circ\W$
if and only if it satisfies
\begin{equation}\label{eq:malcev-basis}
 x^n\eqid x,\qquad (m+1)x\eqid x,\qquad
 mu\eqid mv\quad(u\eqid v\in\Id(\W)).
\end{equation}
For every such $S$, the map
\[
 \eps:S\longrightarrow E^+(S),\qquad \eps(a)=ma
\]
is a retraction with kernel $\Hp$. For its image one has
$I=E^+(S)\in\W$, and each set
\[
 R_i=\eps^{-1}(i)\qquad(i\in I)
\]
is a ring in $\Rn$ with zero element $i$.
\end{lemma}

\begin{proof}
Let $\rho$ witness $S\in\Rn\circ\W$. Each $\rho$-class is a
ring, so $(S,+)$ is a semilattice of abelian groups and hence
a commutative Clifford semigroup. Since the quotient is an
additive semilattice, $\Hp\subseteq\rho$. Conversely, every
$\rho$-class is a subgroup of $(S,+)$ and is contained in a
maximal subgroup, so $\rho\subseteq\Hp$ and $\rho=\Hp$.

By Lemma~\ref{lem:rings}, the first two identities of
\eqref{eq:malcev-basis} hold, and $ma$ is the zero of the ring
component containing $a$, giving
$E^+(S)=\{ma:a\in S\}$. Distributivity and
$(m+1)x\eqid x$ give
\[
 m(a+b)=ma+mb,\qquad (ma)(mb)=m^2ab=mab.
\]
These identities make $\eps$ a homomorphism. It fixes every
additive idempotent and has kernel $\Hp$. The isomorphism
$E^+(S)\cong S/\Hp\in\W$ establishes the remaining identities
in \eqref{eq:malcev-basis}.

Conversely, suppose that \eqref{eq:malcev-basis} holds.
The identity $(m+1)x\eqid x$ makes each cyclic additive
subsemigroup a group with identity $mx$, making $(S,+)$
a commutative Clifford semigroup. The preceding calculation
again shows that $\eps(a)=ma$ is a retraction onto $E^+(S)$.
For every semiring term $u$,
\[
 u(ma_1,\ldots,ma_k)=m\,u(a_1,\ldots,a_k).
\]
The displayed equality shows that $E^+(S)$ satisfies all
identities of $\W$. Its multiplication is idempotent, so
each $\eps$-fiber is closed under multiplication. Such a fiber
has an additive group and satisfies $x^n\eqid x$, making it
a ring in $\Rn$. The congruence $\ker\eps$ therefore witnesses
$S\in\Rn\circ\W$.
\end{proof}

We refer to $I=E^+(S)$ as the \emph{idempotent skeleton} of $S$.
For $a\in S$, we call $\eps(a)$ its \emph{support}.
The additive semilattice order on $I$ will always be
\begin{equation}\label{eq:order}
 i\leq j\quad\Longleftrightarrow\quad i+j=j.
\end{equation}
Multiplication preserves this order in both arguments.
For $e\in I$, write
$\down e=\{i:i\leq e\}$ and $\up e=\{i:e\leq i\}$.

\begin{lemma}\label{lem:connecting}
Let $S\in\Rn\circ\W$ and use the notation of
Lemma~\ref{lem:skeleton}. For every $e\in I$, the set
$S_{\leq e}=\bigcup_{i\leq e}R_i$ is a subsemiring, and
\[
 \alpha_e:S_{\leq e}\longrightarrow R_e,\qquad
 \alpha_e(a)=a+e
\]
is a homomorphism. For $i\leq j$,
\[
 \phi_{i,j}:R_i\longrightarrow R_j,\qquad a\longmapsto a+j
\]
is a ring homomorphism, and
$\phi_{j,k}\phi_{i,j}=\phi_{i,k}$ whenever $i\leq j\leq k$.
\end{lemma}

\begin{proof}
For $i,j\leq e$, the inequalities $i+j\leq e$ and
$ij\leq e^2=e$ show that $S_{\leq e}$ is closed under
both operations.
For $a\in R_i$, the element $ae$ is additively idempotent,
since $e+e=e$. Its support is $ie$, so $ae=ie$.
The right-handed calculation gives $eb=ej$ for $b\in R_j$.
These equalities yield
\[
 (a+e)(b+e)=ab+ie+ej+e=ab+e.
\]
Additive compatibility and the composition formula follow from \eqref{eq:order}.
\end{proof}

Homomorphisms from ring components into an algebraically closed
field extend along connecting homomorphisms under the following
kernel condition; the components need not have multiplicative
identities.

\begin{lemma}\label{lem:extension}
Let $A,B\in\Rn$, and consider a ring homomorphism
$\alpha:A\to B$. Suppose that $\chi:A\to K$ is a nonzero
homomorphism into an algebraically closed field. If
$\ker\alpha\subseteq\ker\chi$, then there is a homomorphism
$\psi:B\to K$ such that $\psi\alpha=\chi$.
\end{lemma}

\begin{proof}
The map $\chi$ induces a homomorphism $\bar\chi:C\to K$,
where $C=\alpha(A)\leq B$. Choose $c\in C$ with
$\bar\chi(c)\ne0$ and put $h=c^{n-1}$. Then $h^2=h$ and
$\bar\chi(h)=1$. The inclusion $hC\subseteq hB$ is an
inclusion of commutative rings with common identity $h$.
It is integral, since every element of $hB$ satisfies the
monic polynomial $X^n-X$ over $hC$.

The image of $hC$ in $K$ is a finite field by
Lemma~\ref{lem:rings}; let $\mathfrak p$ be the kernel of
this map. By the lying-over theorem for integral extensions
\cite[Chapter~5]{AM}, there is a prime ideal
$\mathfrak q$ of $hB$ with
$\mathfrak q\cap hC=\mathfrak p$. The domain
$hB/\mathfrak q$ satisfies $x^n\eqid x$, so it is a finite
field. Its embedding of $hC/\mathfrak p$ extends to an
embedding into $K$, because $K$ is algebraically closed.
This gives a homomorphism $\eta:hB\to K$ extending
$\bar\chi|_{hC}$. Set $\psi(b)=\eta(hb)$. For $c\in C$,
\[
 \psi(c)=\bar\chi(hc)=\bar\chi(c).
\]
This restriction gives $\psi\alpha=\chi$.
\end{proof}

\begin{lemma}\label{lem:separation}
Let $S\in\Rn\circ\W$ have finite skeleton $I$. For distinct
$a,b$ in the same ring component, there is a field $F$,
a homomorphic image of a ring component of $S$, for which
some algebra $Q\in\HS(F\times I)$ admits a homomorphism
$\eta:S\to Q$ satisfying $\eta(a)\ne\eta(b)$.
\end{lemma}

\begin{proof}
Suppose that $a,b\in R_k$. By Lemma~\ref{lem:rings}, there
is a homomorphism $\chi_0:R_k\to K$ into an algebraically
closed field such that $\chi_0(a)\ne\chi_0(b)$. Since $I$
is finite, choose a maximal $e\geq k$ for which there is a
homomorphism $\chi:R_e\to K$ satisfying
$\chi\phi_{k,e}=\chi_0$. For $i\leq e$, put
\[
 \chi_i=\chi\phi_{i,e}.
\]
We claim that
\begin{equation}\label{eq:product-separation}
 ij\leq e\ \hbox{and}\ \chi_{ij}\ne0
 \quad\Longrightarrow\quad i,j\leq e.
\end{equation}
Suppose otherwise, and set $g=ij$ and $f=e+i+j$.
Then $f>e$. Maximality of $e$ implies that $\chi$ does not
extend along $\phi_{e,f}$. By Lemma~\ref{lem:extension},
there is $c\in R_e$ such that
\[
 c+f=f,\qquad \chi(c)\ne0.
\]
Choose $d\in R_g$ such that $\chi(d+e)\ne0$.

In any band, $g=ij$ implies $gig=gjg=g$.
Combining this with the implication from $g\leq e$ to
$g\leq geg\leq e$ gives
\[
 gfg=geg+gig+gjg=geg.
\]
Multiplying $c+f=f$ on both sides by $d$, we obtain
$dcd+dfd=dfd$. Both products belong to the ring $R_{geg}$,
and the additive idempotence of $dfd$ makes $dfd$ the
zero of that ring, giving $dcd=geg$. Applying the homomorphism
$\chi\alpha_e$ from Lemma~\ref{lem:connecting} gives
\[
 \chi(d+e)^2\chi(c)=0,
\]
contradicting the choices of the two nonzero factors.

Put
\[
 D=\down e,\qquad
 U=\{i\leq e:\chi_i\ne0\},\qquad
 V=\up U,\qquad F=\chi(R_e),
\]
where $\up U=\{j\in I:i\leq j\text{ for some }i\in U\}$.
The set $U$ is upward closed within $D$: for
$i\leq j\leq e$, the equality $\chi_i=\chi_j\phi_{i,j}$
expresses the homomorphism at the smaller index as a
composition through the larger one.
It is also closed under multiplication: if $i,j\in U$,
choose $u\in R_i$ and $v\in R_j$ with
$\chi_i(u)\chi_j(v)\ne0$. Then
\[
 \chi_{ij}(uv)=\chi_i(u)\chi_j(v)\ne0.
\]
These closure properties give
\begin{equation}\label{eq:support-filter}
 V\cap D=U,\qquad VV\subseteq V,\qquad V+I\subseteq V.
\end{equation}
The nonzero image $F$ is a finite field in $\Rn$.

Consider the subsemiring
\[
 T=\{(r,i)\in F\times I:i\in V\text{ or }r=0\}.
\]
Closure follows from \eqref{eq:support-filter}. Define an
equivalence relation $\theta$ on $T$ by retaining all
elements of each fiber over $D$, and collapsing each fiber
over $I\setminus D$ to one point, without identifying
different indices. Addition cannot send an index outside
$D$ into $D$, so it preserves $\theta$.
For multiplication, only an index $i\in V\setminus D$
and a factor with nonzero first coordinate require
consideration. The latter factor has index $j\in V$.
If $ij\in D$, then \eqref{eq:support-filter} gives
$ij\in U$, contradicting \eqref{eq:product-separation}.
The argument also applies to $ji$. When the other first
coordinate is zero, both product coordinates are zero,
which gives the remaining compatibility case. These cases
establish $\theta\in\Con(T)$.

Let $Q=T/\theta$ and define, for $s\in R_i$,
\begin{equation}\label{eq:eta}
 \eta(s)=
 \begin{cases}
  [(\chi_i(s),i)]_\theta,&i\leq e,\\
  [(0,i)]_\theta,&i\not\leq e.
 \end{cases}
\end{equation}
This is well defined, since $\chi_i=0$ when $i\in D\setminus U$.
For addition, the condition $i+j\in D$ forces $i,j\in D$,
where Lemma~\ref{lem:connecting} gives compatibility.
Outside this case the target fiber is collapsed.
Multiplicative compatibility for $ij\in D$ and $i,j\in D$
comes from the same lemma. When $ij\in D$ but one of
$i,j$ lies outside $D$, \eqref{eq:product-separation}
implies $\chi_{ij}=0$, so both sides of the required equality
have first coordinate zero. The case $ij\notin D$ has a
collapsed target fiber.

The fiber over $k$ remains uncollapsed and $\chi_k=\chi_0$,
so $\eta(a)\ne\eta(b)$.
\end{proof}

\begin{theorem}\label{thm:join}
For every $\W\leq\Slp$,
\begin{equation}\label{eq:join}
 \Rn\circ\W=\Rn\vee\W.
\end{equation}
For $S\in\Rn\circ\W$ with skeleton $I$ and ring
components $R_i$, one has
\begin{equation}\label{eq:individual-variety}
 \Var(S)=
 \left(\bigvee_{i\in I}\Var(R_i)\right)\vee\Var(I).
\end{equation}
\end{theorem}

\begin{proof}
Assume that $I$ is finite. The map $\eps:S\to I$
separates elements from different ring components.
For a pair in a common component, Lemma~\ref{lem:separation}
supplies a separating homomorphism whose field is a
homomorphic image of a ring component.
Taking the product of these separating homomorphisms embeds
$S$ in a product of algebras belonging to the right-hand
side of \eqref{eq:individual-variety}. The reverse
inclusion follows because every $R_i$ and $I$ are
subsemirings of $S$.

Let $I$ be arbitrary. A finitely generated subsemiring
$S_0$ of $S$ has finitely generated skeleton
$I_0=\eps(S_0)$. Since $\Slp$ is locally finite, $I_0$ is
finite. Its ring components are subrings of the corresponding
components of $S$, and $I_0\leq I$. The finite-skeleton
argument therefore puts every such $S_0$ in the right-hand
side of \eqref{eq:individual-variety}. Since identities
can be checked on finitely generated subalgebras,
\eqref{eq:individual-variety} follows in general.

Equation~\eqref{eq:individual-variety} gives
$\Rn\circ\W\subseteq\Rn\vee\W$.
Lemma~\ref{lem:skeleton} makes $\Rn\circ\W$ a variety
containing both $\Rn$ and $\W$, which gives the reverse
inclusion in \eqref{eq:join}.
\end{proof}

\begin{corollary}\label{cor:ambient-basis}
The variety $\Rn\circ\Slp$ is defined, relative to the
semiring axioms, by
\begin{equation}\label{eq:ambient-basis}
 x^n\eqid x,\qquad (m+1)x\eqid x,\qquad mx^2\eqid mx.
\end{equation}
For $n=2$, the two identities $x^2\eqid x$ and $3x\eqid x$
suffice.
\end{corollary}

\begin{proof}
Under $(m+1)x\eqid x$, the map $x\mapsto mx$ is the
retraction onto the additive idempotents. The last identity
of \eqref{eq:ambient-basis} makes its image multiplicatively
idempotent, so Lemma~\ref{lem:skeleton} identifies the
defined variety. If $n=2$, the last identity is
a consequence of $x^2\eqid x$.
\end{proof}

\section{Subdirectly irreducible semirings}

The restriction to one nontrivial ring component uses the
following idempotent endomorphisms.

\begin{lemma}\label{lem:corners}
Every member of $\Rn\circ\Slp$ satisfies
\begin{equation}\label{eq:corner}
 t^{n-1}xt^{n-1}yt^{n-1}
 \eqid t^{n-1}xyt^{n-1}.
\end{equation}
For every $c\in S$, this identity makes the map
\[
 \pi_h:S\longrightarrow S,\qquad \pi_h(x)=hxh,
 \qquad h=c^{n-1},
\]
an idempotent endomorphism.
\end{lemma}

\begin{proof}
In $\Rn$, commutativity and the multiplicative
idempotence of $t^{n-1}$ establish the identity.
Its restriction to $\Slp$ is \eqref{eq:regular-band}.
Theorem~\ref{thm:join} gives \eqref{eq:corner}.
Distributivity proves additive compatibility of $\pi_h$,
and
\[
 (hxh)(hyh)=hxh yh=hxyh
\]
proves multiplicative compatibility. The equality $h^2=h$
gives $\pi_h^2=\pi_h$.
\end{proof}

\begin{proposition}\label{prop:one-field}
Let $S\in\Rn\circ\W$ be subdirectly irreducible and not
additively idempotent. Then $S$ has a multiplicative identity,
and its monolith is $\Hp$. Exactly one ring component of
$S$ is nontrivial. This component is a finite field
$\F_q$, where $q-1\mid n-1$.
\end{proposition}

\begin{proof}
Let $\mu$ be the monolith of $S$. Since $\Hp$ is
nontrivial, $\mu\subseteq\Hp$. Choose $a\ne b$ with
$a\mathrel\mu b$. These elements belong to a common
component $R_e$. In that ring, put
$c=a-b\ne e$ and $h=c^{n-1}$. Since $hch=c$, the
endomorphism $\pi_h$ from Lemma~\ref{lem:corners}
separates $a$ and $b$, so $\mu$ is not contained in
$\ker\pi_h$. Subdirect irreducibility gives
$\ker\pi_h=\Delta_S$. An injective idempotent endomorphism
is the identity, so $hxh=x$ for every $x\in S$.
Multiplying this equality by $h$ on either side gives
$hx=xh=x$.

From $a\mathrel\mu b$, addition of $-b$ in $R_e$
gives $c\mathrel\mu e$. Taking $(n-1)$st powers yields
$h\mathrel\mu e$. Since $e=mh$, every $s\in S$ satisfies
\[
 s=sh\mathrel\mu se=s(mh)=ms=\eps(s).
\]
This gives $\Hp\subseteq\mu$; together with the inclusion
established at the start of the proof, it yields
\begin{equation}\label{eq:monolith}
 \mu=\Hp.
\end{equation}

Suppose that $R_f$ is any nontrivial ring component and
choose $d\ne f$ in $R_f$. The endomorphism
$\pi_{d^{n-1}}$ separates $d$ and $f$. By
\eqref{eq:monolith}, its kernel cannot be nontrivial.
As before, $d^{n-1}$ is a multiplicative identity of $S$,
and uniqueness gives $d^{n-1}=h$. Taking supports yields
$f=e$, so $R_e$ is the unique nontrivial ring component.

Every nonzero $d\in R_e$ satisfies $d^{n-1}=h$,
which makes $R_e$ a field. Lemma~\ref{lem:rings}
identifies it as $R_e\cong\F_q$ with $q-1\mid n-1$.
\end{proof}

For the semirings in Proposition~\ref{prop:one-field},
the mixed operations between the field component and
the idempotent skeleton remain to be specified.

\begin{definition}\label{def:construction}
Let $F\in\Rn$ be a finite field, and let $I\in\Slp$
have a multiplicative identity $e$. Define $A(F,I,e)$
on the set $F\cup I$, with $F\cap I=\{0_F\}=\{e\}$.
Retain the operations on $F$ and $I$, and prescribe
the mixed operations by
\begin{equation}\label{eq:mixed-product}
 ri=ir=i\qquad(r\in F,\ i\in I)
\end{equation}
and
\begin{equation}\label{eq:mixed-sum}
 r+i=i+r=
 \begin{cases}
  r,&i\leq e,\\
  e+i,&i\not\leq e.
 \end{cases}
\end{equation}
\end{definition}

\begin{lemma}\label{lem:construction}
The operations in Definition~\ref{def:construction} make
$A(F,I,e)$ a semiring in $\HS(F\times I)$. Its skeleton
is $I$, its unique nontrivial ring component is $F$,
and
\begin{equation}\label{eq:construction-variety}
 \Var(A(F,I,e))=\Var(F)\vee\Var(I).
\end{equation}
The multiplicative identity of $A(F,I,e)$ is $1_F$.
\end{lemma}

\begin{proof}
Consider
\[
 T=(F\times\up e)\ \cup\ (\{0_F\}\times I).
\]
The order is compatible with multiplication and $e^2=e$,
so $\up e$ is closed under multiplication and is upward
closed under addition. These closure properties give
$T\leq F\times I$.
On $T$, collapse each fiber with index different from $e$
to one point, retaining every element of the fiber over $e$.

This equivalence is a congruence. A fiber which is
nontrivially collapsed has index $i>e$, and addition cannot
send this index back to $e$. For multiplication, a band with
identity $e$ satisfies the implication from $ij=e$ to $i=j=e$:
$i=i(ij)=i^2j=ij=e$, and then $j=e$.
Multiplication cannot distinguish elements of a collapsed
fiber in the fiber over $e$. All other target
fibers are singletons.

The resulting quotient is the algebra in
Definition~\ref{def:construction}. The two definitions
of its operations agree at the common element $e=0_F$.
The embedded copies of $F$ and $I$ give one inclusion in
\eqref{eq:construction-variety}; membership in
$\HS(F\times I)$ gives the other.
Equations~\eqref{eq:mixed-product} and \eqref{eq:mixed-sum}
identify the skeleton, the ring component, and the
multiplicative identity.
\end{proof}

\begin{lemma}\label{lem:reconstruction}
Every subdirectly irreducible semiring $S\in\Rn\circ\W$
that is not additively idempotent is isomorphic to $A(F,I,e)$
for some
$F=\F_q$, $q-1\mid n-1$, and $I\in\W$ with multiplicative
identity $e$.
\end{lemma}

\begin{proof}
By Proposition~\ref{prop:one-field}, $S$ has a unique
nontrivial ring component $F=R_e$ and a multiplicative
identity $h=1_F$. All other ring components are singletons.
Its underlying set is $F\cup I$, with
$F\cap I=\{e\}$. For $i\in I$,
\[
 ei=(mh)i=m(hi)=mi=i,
\]
and the calculation on the other side gives $ie=i$.
Together these identify $e$ as the multiplicative identity
of $I$.

For $r\in F$ and $i\in I$, the product $ri$ is
additively idempotent with support $ei=i$, which gives
$ri=i$. Reversing the multiplication gives $ir=i$.
If $i\leq e$, then
$r+i=r+e+i=r+e=r$. Otherwise $r+i$ belongs to the
singleton component with support $e+i\ne e$, so
$r+i=e+i$.
\end{proof}

For an algebra $I$, let $\Pol_1(I)$ denote its unary
polynomial functions, allowing parameters from $I$.
If $D\subseteq I$, define
\begin{equation}\label{eq:syntactic}
 i\equiv_D j
 \quad\Longleftrightarrow\quad
 \bigl(p(i)\in D\ \Longleftrightarrow\ p(j)\in D\bigr)
 \text{ for every }p\in\Pol_1(I).
\end{equation}
This is the \emph{syntactic congruence} of $D$ on $I$;
see~\cite[Lemma~3.1]{AGK} for the equivalent formulation
using linear unary polynomials.

\begin{lemma}\label{lem:syntactic}
The relation $\equiv_D$ is the largest congruence of $I$
for which $D$ is a union of congruence classes.
\end{lemma}

\begin{proof}
The relation defined by the polynomial condition is an
equivalence relation.
The identity polynomial ensures that it saturates $D$,
and composition preserves it under every unary polynomial.
For a basic binary operation, replacing the two arguments
one at a time, with the other fixed as a parameter,
establishes compatibility, making $\equiv_D$ a congruence.
Any congruence $\tau$ saturating $D$ preserves membership
after polynomial evaluation: $i\mathrel\tau j$ implies
$p(i)\mathrel\tau p(j)$ for every unary polynomial $p$,
and hence $i\equiv_D j$.
\end{proof}

\begin{theorem}\label{thm:criterion}
Let $S=A(F,I,e)$, where $F\in\Rn$ is a finite field and
$I\in\Slp$ has identity $e$. Put $D=\down e$. Then
\begin{equation}\label{eq:si-criterion}
 S\text{ is subdirectly irreducible}
 \quad\Longleftrightarrow\quad
 \equiv_D\,=\,\Delta_I.
\end{equation}
When these equivalent conditions hold, the monolith
has $F$ as its unique nonsingleton class, and
\begin{equation}\label{eq:congruence-lattice}
 \Con(S)=\{\Delta_S\}\ \cup\
 \{\eps^{-1}(\tau):\tau\in\Con(I)\},
\end{equation}
where $\eps^{-1}(\tau)=
\{(a,b):(\eps(a),\eps(b))\in\tau\}$.
\end{theorem}

\begin{proof}
Let $H=\ker\eps$. Its only nonsingleton class is $F$.
A congruence $\theta$ that is nontrivial on $F$ identifies
all of $F$ by simplicity of the field, so $H\subseteq\theta$.

Suppose instead that $\theta|_F=\Delta_F$.
No $r\in F\setminus\{e\}$ can be $\theta$-equivalent to
an $i\in I$. Otherwise application of the unary term
$mx$ gives $e\mathrel\theta i$, and hence
$r\mathrel\theta e$, a contradiction. It follows from this
exclusion that
\[
 \theta=\tau\ \cup\ \Delta_{F\setminus\{e\}}
\]
for a congruence $\tau$ on $I$. This congruence must
saturate $D$. For if $i\mathrel\tau j$ with $i\in D$
and $j\notin D$, adding $1_F$ gives
$1_F\mathrel\theta e+j$, again identifying a nonzero
field element with an element of $I$.

Conversely, let $\tau\in\Con(I)$ saturate $D$.
Extend it by singleton classes on $F\setminus\{e\}$.
For $i\mathrel\tau j$, mixed multiplication preserves
the relation by \eqref{eq:mixed-product}. For addition
of $r\in F$, if $i,j\in D$ both sums equal $r$; if
$i,j\notin D$, the sums $e+i$ and $e+j$ are
$\tau$-equivalent. The extension is a congruence
of $S$ whose intersection with $H$ is $\Delta_S$.

By Lemma~\ref{lem:syntactic}, there is no nontrivial
congruence of this second kind exactly when
$\equiv_D=\Delta_I$. In that case every nontrivial
congruence contains $H$, so $H$ is the monolith.
If a nontrivial congruence of the second kind exists,
its intersection with $H$ is trivial, so $S$ is not
subdirectly irreducible, establishing both directions of
\eqref{eq:si-criterion}. Formula~\eqref{eq:congruence-lattice}
then follows from the correspondence theorem
\cite[Chapter~II]{BS} applied to $S/H\cong I$.
\end{proof}

\begin{theorem}\label{thm:classification}
Let $\W\leq\Slp$. Up to isomorphism, the subdirectly
irreducible members of $\Rn\circ\W$ are precisely:
\begin{enumerate}
 \item the subdirectly irreducible members of $\W$;
 \item the semirings $A(\F_q,I,e)$, where $q-1\mid n-1$,
 $I\in\W$ has multiplicative identity $e$, and for all
 distinct $i,j\in I$ there is a $p\in\Pol_1(I)$ such that
 exactly one of $p(i),p(j)$ belongs to $\down e$.
\end{enumerate}
There is no finiteness restriction on $I$.
The parameters in {\rm(2)} are unique up to isomorphism:
$A(\F_q,I,e)\cong A(\F_{q'},I',e')$ if and only if
$q=q'$ and $I\cong I'$.
\end{theorem}

\begin{proof}
An additively idempotent member of $\Rn\circ\W$ coincides
with its skeleton and hence belongs to $\W$.
For all other members, necessity follows from
Lemma~\ref{lem:reconstruction} and
Theorem~\ref{thm:criterion}. Conversely,
Lemma~\ref{lem:construction} puts every algebra in
(2) in $\Rn\vee\W$, and
Theorem~\ref{thm:criterion} makes it subdirectly
irreducible.

An isomorphism preserves additive idempotents and ring
components, so it preserves both $I$ and the unique
nontrivial component $\F_q$. Conversely, isomorphisms
of the two parameters agree at their identified zero
and identity elements, and extend to an isomorphism
by \eqref{eq:mixed-product} and \eqref{eq:mixed-sum}.
\end{proof}

\section{Subvariety lattices and finite bases}

\begin{theorem}\label{thm:lattice}
For every $\W\leq\Slp$, the map
\[
 L(\Rn\circ\W)\longrightarrow L(\Rn)\times L(\W),
 \qquad
 \V\longmapsto(\V\cap\Rn,\V\cap\W)
\]
is a lattice isomorphism. Its inverse sends
$(\A,\B)$ to $\A\vee\B$.
\end{theorem}

\begin{proof}
Let $\V\leq\Rn\circ\W$. For each $S\in\V$,
all its ring components belong to $\V\cap\Rn$, and
its skeleton belongs to $\V\cap\W$.
Equation~\eqref{eq:individual-variety} gives
\begin{equation}\label{eq:subvariety-decomp}
 \V=(\V\cap\Rn)\vee(\V\cap\W).
\end{equation}

It remains to recover an arbitrary pair
$\A\leq\Rn$, $\B\leq\W$ from $\A\vee\B$.
If $u\eqid v$ holds in $\A$, then
\begin{equation}\label{eq:ring-lift}
 u+mv\eqid mu+v
\end{equation}
holds in both $\A$ and every member of $\Slp$.
On $\Rn$, this identity is equivalent to $u\eqid v$.
These identities recover $(\A\vee\B)\cap\Rn=\A$.
For an identity $u\eqid v$ of $\B$, the equation
\begin{equation}\label{eq:idempotent-lift}
 mu\eqid mv
\end{equation}
holds in $\B$ and $\Rn$ and reduces to $u\eqid v$
on $\Slp$, recovering $(\A\vee\B)\cap\W=\B$.
The maps are inverse order-preserving bijections
and are lattice isomorphisms.
\end{proof}

For a prime $p$, let
\[
 T_p(n)=\{t\geq1:p^t-1\mid n-1\},
\]
ordered by divisibility, and let $\mathcal O(T_p(n))$
be its lattice of downward closed subsets.
Only finitely many of these posets are nonempty.

\begin{proposition}\label{prop:ring-lattice}
There is a lattice isomorphism
\begin{equation}\label{eq:ring-lattice}
 L(\Rn)\cong\prod_{p\ {\rm prime}}\mathcal O(T_p(n)).
\end{equation}
The coordinate associated with $\A\leq\Rn$ at $p$ is
$\{t:\F_{p^t}\in\A\}$.
\end{proposition}

\begin{proof}
By Lemma~\ref{lem:rings}, each subvariety of $\Rn$
is generated by the finite fields which it contains.
For a fixed characteristic, these fields form a
downward closed set, because
$\F_{p^s}\leq\F_{p^t}$ if and only if $s\mid t$
\cite[Chapter~2]{LN}.

To realize prescribed downward closed sets without adding
fields, choose $D_p\subseteq T_p(n)$ and put
$P=\{p:D_p\ne\varnothing\}$. Assume $P\ne\varnothing$. Put $d=\prod_{p\in P}p$, and let
$M_p$ be the set of maximal elements of $D_p$.
The variety generated by the selected fields satisfies
the ring identities
\begin{equation}\label{eq:field-selection}
 dx=0,\qquad
 \frac{d}{p}\prod_{s\in M_p}(x^{p^s}-x)=0
 \quad(p\in P).
\end{equation}
Here ring notation is an abbreviation:
within $\Rn$, replace $-x$ by $(m-1)x$ and zero by $mx$;
the first identity may be written $(d+1)x\eqid x$.

Let $\F_{p^t}\in\Rn$ satisfy \eqref{eq:field-selection}.
The first identity forces $p\in P$.
Choose a generator $a$ of its multiplicative group.
The scalar $d/p$ is nonzero in characteristic $p$,
so one factor in the second identity vanishes at $a$.
For some $s\in M_p$,
$a^{p^s}=a$, whence $p^t-1\mid p^s-1$.
This is equivalent to $t\mid s$, and therefore $t\in D_p$.
No unselected field can belong to the generated variety.
If $P$ is empty, take the trivial variety.
Lemma~\ref{lem:rings} proves bijectivity; inclusion of the
selected fields also preserves the order.
\end{proof}

\begin{corollary}\label{cor:count}
The variety $\Rn\circ\Slp$ has
\[
 78\prod_{p\ {\rm prime}}|\mathcal O(T_p(n))|
\]
subvarieties. For $n=2,3,4$, these numbers are
$156$, $312$, and $234$, respectively.
\end{corollary}

\begin{proof}
Ghosh, Pastijn and Zhao \cite{GPZ} proved that
$|L(\Slp)|=78$. Apply Theorem~\ref{thm:lattice} and
Proposition~\ref{prop:ring-lattice}. When $n=2$,
only $\F_2$ occurs. The case $n=3$ allows
$\F_2$ and $\F_3$, which have different characteristics.
With $n=4$, the permitted fields are $\F_2$ and $\F_4$;
the former is a subfield of the latter.
The difference between the counts for $n=3$ and $n=4$ shows
that the enumeration depends on the subfield relations among
the permitted fields, even when their number is fixed.
\end{proof}

\begin{lemma}\label{lem:finite-lattice}
A finitely based variety with finitely many subvarieties
is hereditarily finitely based.
\end{lemma}

\begin{proof}
Let $\V$ have a finite identity basis $\Sigma$, and
let $\Upsilon\leq\V$. For each $\mathbf U\leq\V$ with
$\mathbf U\not\subseteq\Upsilon$, choose an identity
holding in $\Upsilon$ and failing in $\mathbf U$.
There are only finitely many choices. Adjoin them
to $\Sigma$. The resulting variety contains $\Upsilon$
and is contained in $\V$. If it were not contained in
$\Upsilon$, one of the chosen identities would exclude
it, a contradiction.
\end{proof}

\begin{theorem}\label{thm:finite-basis}
For every $n\geq2$ and every $\W\leq\Slp$, the variety
$\Rn\circ\W$ is hereditarily finitely based.
\end{theorem}

\begin{proof}
Corollary~\ref{cor:ambient-basis} gives a finite basis
for $\Rn\circ\Slp$, and Corollary~\ref{cor:count}
shows that its subvariety lattice is finite.
Lemma~\ref{lem:finite-lattice} applies to this variety
and all its subvarieties.
\end{proof}

The two identity transformations used in
Theorem~\ref{thm:lattice} give bases for individual joins.

\begin{proposition}\label{prop:basis-transfer}
Let $\A\leq\Rn$ be defined relative to $\Rn$ by a set
$\Sigma_{\A}$, and let $\B\leq\Slp$ be defined relative
to $\Slp$ by $\Sigma_{\B}$. Then $\A\vee\B$ is defined
relative to the semiring axioms by
\eqref{eq:ambient-basis}, together with
\[
 \begin{aligned}
 u+mv&\eqid mu+v
       &&(u\eqid v\in\Sigma_{\A}),\\
 mu&\eqid mv
       &&(u\eqid v\in\Sigma_{\B}).
 \end{aligned}
\]
Finite choices of $\Sigma_{\A}$ and $\Sigma_{\B}$
give a finite basis for $\A\vee\B$.
\end{proposition}

\begin{proof}
These identities hold in both $\A$ and $\B$.
Conversely, suppose that $S$ satisfies them.
By Corollary~\ref{cor:ambient-basis},
$S\in\Rn\circ\Slp$. Restricting the first family to a
ring component yields $\Sigma_{\A}$ and places the component
in $\A$. The second family induces $\Sigma_{\B}$ on the
skeleton, which therefore belongs to $\B$.
Equation~\eqref{eq:individual-variety} gives
$S\in\A\vee\B$.
\end{proof}

\section{Special cases and unbounded subdirectly irreducible members}

Under the absorption identity, the classification reduces
to the field and idempotent cases.

\begin{corollary}\label{cor:absorption}
If $\W\leq\mathbf N^+$, then the subdirectly irreducible
members of $\Rn\circ\W$ are precisely the finite fields
$\F_q$ with $q-1\mid n-1$ and the subdirectly
irreducible members of $\W$. Every member of
$\Rn\circ\W$ is a subdirect product of a ring in $\Rn$
and a semiring in $\W$.
\end{corollary}

\begin{proof}
If $I\in\W$ has multiplicative identity $e$, substitution
of $x=e$ and $y=i$ in $x+xyx\eqid x$ gives
$e+i=e$ for every $i\in I$, giving $\down e=I$.
The relation in \eqref{eq:syntactic} is then the
universal relation, and it is equality only when
$I$ is trivial. Theorem~\ref{thm:classification}
therefore gives the stated list.

Represent an arbitrary member as a subdirect product
of subdirectly irreducible members, and group the
field factors and the idempotent factors separately.
The two projection images belong to $\Rn$ and $\W$,
respectively, and give the required subdirect product.
An empty group of factors is replaced by the trivial
semiring.
\end{proof}

\begin{remark}
For $\W=\mathbf N^+$, Corollary~\ref{cor:absorption}
recovers the decomposition theorem of Wang and Shao
\cite[Theorem~2.1]{WS}. Theorem~\ref{thm:join}
extends their join equality, while
Theorem~\ref{thm:classification} describes the
additional members which can occur when the absorption
identity is removed.
\end{remark}

Let $\mathbf{Sl}$ be the variety in which addition and
multiplication coincide and form a semilattice.
For a finite field $F$, let $F^\infty$ denote the
extension by an element $\infty$ satisfying
$r+\infty=\infty+r=\infty$ and
$r\infty=\infty r=\infty$ for all $r\in F^\infty$.

\begin{corollary}\label{cor:semilattice}
The subdirectly irreducible members of
$\Rn\circ\mathbf{Sl}$ are the two-element semilattice
semiring and the semirings $\F_q$ and $\F_q^\infty$,
where $q-1\mid n-1$.
\end{corollary}

\begin{proof}
If $I\in\mathbf{Sl}$ has multiplicative identity $e$,
then $e$ is its least element and $\down e=\{e\}$.
A unary polynomial which depends on its variable is
$x$ or $x+c$. These functions make all elements of
$I\setminus\{e\}$ equivalent under $\equiv_{\{e\}}$,
while the identity function separates $e$ from each of them.
For this equivalence relation, the criterion in
Theorem~\ref{thm:criterion} holds exactly when
$|I|\leq2$, yielding $F$ and $F^\infty$.

A semilattice embeds in a power of
the two-element semilattice: if $a\not\leq b$, the map
which is zero on $\down b$ and one elsewhere is a
join homomorphism separating $a$ and $b$.
In a subdirectly irreducible semilattice, a coordinate
which separates a monolith pair must have trivial
kernel, forcing a nontrivial such semilattice to have
exactly two elements.
\end{proof}

A variety is \emph{residually large} if the cardinalities
of its subdirectly irreducible members have no uniform bound.

\begin{theorem}\label{thm:large}
There is a four-element semiring $B$ in
$\mathbf R_2\circ\Slp$ such that $\Var(B)$ is
residually large. The variety $\Rn\circ\Slp$ is
residually large for every $n\geq2$.
\end{theorem}

\begin{proof}
On the chain $K=\{0,e,t\}$ with $0<e<t$, define addition
by maximum. Multiplication has absorbing element $0$ and
agrees with maximum on $\{e,t\}$.
These operations are associative, and multiplication
is order preserving on a chain, hence distributes
over maximum. This gives $K\in\Slp$, with $e$ as its
multiplicative identity.

For a nonempty set $J$, put
\[
 I_J=K^J,\qquad e_J=(e)_{j\in J},\qquad
 D_J=\down e_J=\{0,e\}^J.
\]
For the separation condition in
Theorem~\ref{thm:classification}, take distinct
$a,b\in I_J$ and choose $j$ with $a_j\ne b_j$.
Let $c$ be zero at every coordinate other than $j$.
If $\{a_j,b_j\}=\{0,e\}$, put $c_j=t$.
Then the $j$th coordinates of $ac$ and $bc$ are
$0$ and $t$, in some order. If one of $a_j,b_j$
is $t$, put $c_j=e$. Again exactly one of the two
products has coordinate $t$ at $j$, and all their
other coordinates are zero. The unary polynomial
$p(x)=xc$ sends precisely one of $a,b$ into $D_J$.

The separation criterion makes
\[
 S_J=A(\F_2,K^J,e_J)
\]
subdirectly irreducible. By
\eqref{eq:construction-variety},
\[
 \Var(S_J)\subseteq\Var(\F_2)\vee\Var(K)
 =\Var\bigl(A(\F_2,K,e)\bigr).
\]
Set $B=A(\F_2,K,e)$. It has four elements.
For $|J|=k<\infty$, the algebra $S_J$ has $3^k+1$
elements; for infinite $J$, it has $3^{|J|}$ elements.
The absence of a uniform bound on these cardinalities makes
$\Var(B)$ residually large.
The membership $\F_2\in\Rn$ for every $n\geq2$ also
gives $B\in\Rn\vee\Slp$ for every such $n$.
\end{proof}

Writing $u$ for the nonzero element of $\F_2$, the
semiring $B$ has operation tables
\[
\begin{array}{c|cccc}
 +&0&e&u&t\\ \hline
 0&0&e&u&t\\
 e&e&e&u&t\\
 u&u&u&e&t\\
 t&t&t&t&t
\end{array}
\qquad
\begin{array}{c|cccc}
 \cdot&0&e&u&t\\ \hline
 0&0&0&0&0\\
 e&0&e&e&t\\
 u&0&e&u&t\\
 t&0&t&t&t
\end{array}.
\]
Since the additive identity of $B$ is $0$, the entries
$e+e=e\ne0$ and $u+u=e\ne u$ show that $B$ is neither a
ring nor additively idempotent.
Its field component has zero $e$, and $u$ is the multiplicative
identity of $B$.

\begin{corollary}
A variety generated by a finite semiring can have both
a finite subvariety lattice and a finite basis for each
subvariety while admitting subdirectly irreducible members
of unbounded cardinality.
\end{corollary}

\begin{proof}
Take $\Var(B)$ from Theorem~\ref{thm:large}.
It is a subvariety of $\mathbf R_2\circ\Slp$, so the
first two additional properties follow from
Corollary~\ref{cor:count} and
Theorem~\ref{thm:finite-basis}.
\end{proof}

\begin{remark}\label{rem:no-two-factor}
The semiring $B$ generating the variety in the preceding
corollary is subdirectly irreducible and is neither a ring nor
additively idempotent.
If it admitted a subdirect embedding into a product of a ring
in $\Rn$ and an idempotent semiring, the coordinate kernels
would have trivial intersection. Subdirect irreducibility forces
one kernel to be trivial. The corresponding projection is
surjective, so $B$ would be isomorphic to that factor, a
contradiction. This counterexample shows that the join equality
of Theorem~\ref{thm:join} does not imply the two-factor
decomposition of Corollary~\ref{cor:absorption} for arbitrary
$\W\leq\Slp$.
\end{remark}


\begin{thebibliography}{99}

\bibitem{AGK}
J. Almeida, H. Goulet-Ouellet and O. Kl\'ima,
What makes a Stone topological algebra profinite,
\emph{Algebra Universalis} \textbf{84} (2023), Article~6.
\href{https://doi.org/10.1007/s00012-023-00804-w}
{doi:10.1007/s00012-023-00804-w}.

\bibitem{AM}
M. F. Atiyah and I. G. Macdonald,
\emph{Introduction to Commutative Algebra},
Addison-Wesley, Reading, Massachusetts, 1969.

\bibitem{BS}
S. Burris and H. P. Sankappanavar,
\emph{A Course in Universal Algebra},
Springer-Verlag, New York--Berlin, 1981.

\bibitem{GPZ}
S. Ghosh, F. Pastijn and X. Z. Zhao,
Varieties generated by ordered bands. I,
\emph{Order} \textbf{22} (2005), 109--128.
\href{https://doi.org/10.1007/s11083-005-9011-z}
{doi:10.1007/s11083-005-9011-z}.

\bibitem{Howie}
J. M. Howie,
\emph{Fundamentals of Semigroup Theory},
Clarendon Press, Oxford University Press, 1995.

\bibitem{Jacobson}
N. Jacobson,
Structure theory for algebraic algebras of bounded degree,
\emph{Ann. of Math. (2)} \textbf{46} (1945), 695--707.

\bibitem{LN}
R. Lidl and H. Niederreiter,
\emph{Finite Fields}, 2nd ed.,
Encyclopedia of Mathematics and its Applications, Vol.~20,
Cambridge University Press, Cambridge, 1997.
\href{https://doi.org/10.1017/CBO9780511525926}
{doi:10.1017/CBO9780511525926}.

\bibitem{Malcev}
A. I. Mal'cev,
Multiplication of classes of algebraic systems,
\emph{Siberian Math. J.} \textbf{8} (1967), 254--267.
\href{https://doi.org/10.1007/BF02302476}
{doi:10.1007/BF02302476}.

\bibitem{PZ}
F. Pastijn and X. Zhao,
Varieties of idempotent semirings with commutative addition,
\emph{Algebra Universalis} \textbf{54} (2005), 301--321.
\href{https://doi.org/10.1007/s00012-005-1947-8}
{doi:10.1007/s00012-005-1947-8}.

\bibitem{WS}
A. Wang and Y. Shao,
On a semiring variety satisfying $x^n\approx x$,
\emph{Publ. Math. Debrecen} \textbf{93} (2018), 73--86.
\href{https://doi.org/10.5486/PMD.2018.8022}
{doi:10.5486/PMD.2018.8022}.

\end{thebibliography}
\end{document}